\documentclass[12 pt,aps,amssymb] {amsart}
\usepackage{amsmath,amsthm,amssymb}
\newtheorem{theorem}{Theorem}
\newtheorem{proposition}[theorem]{Proposition}
\newtheorem{lemma}[theorem]{Lemma}

 \theoremstyle{definition}
\newtheorem{definition}[theorem]{Definition}

\begin{document}
\title {A parabolic action on a proper, CAT(0) cube complex}\par\smallskip

\author {Yael Algom-Kfir, Bronislaw Wajnryb, Pawel Witowicz}
\begin{abstract} We consider diagram groups as defined by V. Guba and M. Sapir in {\em Diagram groups}, Mem. Amer. Math. Soc. 130 (1997). A diagram group $G$ acts on the associated cube complex $K$ by isometries. It is known that if a cube complex $L$ is of a finite dimension then every isometry $ g$ of $L$ is semi-simple - $inf\{d(x,gx): x\in L\}$ is attained. It was conjectured by D. S. Farley that in the case of a diagram group $G$ the action of $G$ on the associated cube complex $K$ is by semisimple isometries also when $K$ has an infinite dimension. In this paper we give a counterexample to Farley Conjecture and we show that  R. Thompson's group $F$, considered as a diagram group, has some elements which act as parabolic (not semi-simple) isometries on the associated cube complex.
\end{abstract}
 \maketitle

\renewcommand{\thefootnote}{}
\footnote{2010\emph{Mathematics Subject Classification}: Primary  20F65, 20F67.}

\section{Introduction}

Isometries of metric spaces can be classified according to their translation length. If $g$ is an isometry of $X$ then its translation length is  $|g| = \inf \{ d(x,gx) \mid x \in X \}$. We distinguish between two cases:
\begin{itemize}
	\item $g$ is {\it semi-simple} when the infimum in the definition of $g$ is realized as a minimum. $g$ is {\it elliptic} when $|g|=0$ and $g$ is {\it hyperbolic} when $|g|>0$. 
	\item If the infimum is not realized $g$ is called {\it parabolic}.
\end{itemize}

 Bridson proved in \cite{B} that if $X$ is an $M_\kappa$ polyhedral complex where the number of isometry types of polygons ({\em shapes}) is finite, all isometries are semi-simple. In particular it is true if $X$ is a finite-dimensional cube complex where we identify each cube with the unit cube in the Euclidean space and we consider the induced length metric in the complex.

Hilbert Space $\mathit{l}^2(\mathbf{Z})$ has a structure of a CAT(0) cube complex and it admits a parabolic isometry (see \cite{BH}, Example II.8.28) but it is not a proper space and it is not locally of a finite dimension.

 Are all isometries of a proper CAT(0) cube complex semi-simple? 

If $G=D(P,w)$ is a diagram group (see the next section) corresponding to a given 
finite presentation of a semigroup $P$ and to a word $w$ in $P$ then there is a proper, locally finite CAT(0) cube complex $K(P,w)$ associated to $G$ and $G$ acts freely on $K(P,w)$ by cellular isometries. 
Is this action by semi-simple isometries? 

This question was explicitly asked by 
Farley \cite{farley}, Question 3.17, who suggested that the answer is affirmative.  The problem is known as Farley Conjecture\par\smallskip
{\bf Farley Conjecture.} For every $P$ and $w$ the diagram group $D(P,w)$ acts on $K(P,w)$ by semi-simple isometries.\par\smallskip

 In this paper we give a counter-example to Farley Conjecture. There is a diagram group  $G=D(P,w)$ and an element $g\in G$ which acts as a parabolic isometry on the proper CAT(0) cube complex $K(P,w)$. Moreover, $|g|$ cannot be approximated by a sequence of centers of cubes. \\

The example arises from the representation of R. Thompson's group $F$ as a diagram group. The diagram in Figure 1 (left side) represents the element of Thompson's group which acts as a parabolic isometry on the associated cube complex. \par\medskip

The paper is organized as follows. In Section 2 we recall the notions of diagrams over a given presentation of a semi-group $P$, the associated cube complex $K(P,w)$, the diagram group $D(P,w)$ and their basic properties. We formulate and explain the main result: Theorem 1.  In Section 3 we describe  some finite convex subcomplexes of $K(P,w)$ (Theorem \ref{convex}). In Section 4 we describe precisely the action of an element of $D(P,w)$ restricted to a given cube of $K(P,w)$. In Section 5 we restrict the investigation to the R. Thompson's group $F$ and we prove that the element $g$ of the group $F$ shown on Figure \ref{diagramg} acts as a parabolic isometry on $K(P,w)$.

\section{Preliminaries}
We consider plane diagrams and diagram groups as in \cite{GS} and in \cite{farley}. We recall the basic notions.
Let $A$ be a finite set of letters and let $S_A$ be the free semi-group on
letters from $A$. Let $R$ be some finite set of pairs of words $(u,v)$ from
$S_A$ such that $u\ne v$ and if $(u,v)\in R$ then $(v,u)\in R$.

Let $P$ be the semigroup with the set of generators $A$ and the set of
relations  $\{u=v: (u,v)\in R\}$. We consider diagrams over the given presentation $\langle A|R \rangle $ of the semigroup $P$.
 The simplest type of diagrams, a $degenerate$ diagram, say $U_0$, consists of one path which leads monotonically from left to right. It has the left end and the right end and is composed of a finite number of edges, which are labelled by letters of $A$ and form a word $w$. We write $top(U_0)=bot(U_0)=w$. Another simple diagram is a cell $C$. It is a disk bounded by two paths with common end points, the {\em top path} $\overline C$  and the {\em bottom path} $\underline C$. The paths have the left end and the right end and are monotonic in $x$-coordinate and are oriented from left to right. Each path is composed of a finite number of edges which are labelled by letters of $A$. The letters read from left to right form words $u=top(C)$ and $v=bot(C)$ and the pair $(u,v)$ must belong to the set $R$. Then the cell $C$ is denoted $C(u,v)$.

Suppose $U_0$ is a degenerate diagram with $top(U_0)=w$. Let  $(u,v)$ be a pair in $R$ and suppose that $w$ contains a subword $u$. We can attach a cell $C(u,v)$ to $U_0$ from below along the suitable sub-arc $e$ ( the union of consecutive edges) with the label $u$. We get a new diagram, say $U_1$. It has the {\em top path} $\overline U_1$ with the label $top(U_1)=w$ and the {\em bottom path} $\underline U_{\,1}$ with the label $bot(U_1)=w_1$, where $w_1$ is obtained from $w$ by the replacement of the suitable subword $u$ by $v$.

We may attach another cell from below to $U_1$ along some sub-arc of $\underline U_{\,1}$ and so on. Every diagram is obtained from a degenerate diagram by a consecutive attachment of a finite number of cells from below. An {\em atomic}  diagram is a diagram with only one cell.

If $U$ is a diagram then the $inverse$ diagram $U^{-1}$ is obtained from $U$ by the reflection with respect to a horizontal line. In particular $top(U^{-1})=bot(U)$ and $bot(U^{-1})=top(U)$.

 Isotopic diagrams with the same labels on the corresponding edges are considered equal. The same diagram may be obtained by attaching cells in different order.  If a diagram contains two cells $C_1,C_2$ such that $C_1=C(u,v)$ and $C_2=C(v,u)$ and $ \underline C_{\,1}=\overline C_2$ then we can cancel $C_1$ and $C_2$ in the diagram removing the middle path  $ \underline C_{\,1}=\overline C_2$, contracting the interiors  and identifying $\overline C_1$ with $\underline C_{\,2}$. A diagram is $reduced$ if no cancelation is possible. A suitable sequence of cancelations will take every diagram to a reduced diagram and the result is independent of the order of the cancelations (see \cite{kilibarda} and \cite{GS}, Theorem 3.17).

If $U,V$ are diagrams with $bot(U)=top(V)$ we can compose $U$ and $V$ by concatenation (attaching $V$ under $U$). We denote the result by $U\circ V$. If the result is not reduced we may reduce it. The reduced diagram is denoted by $UV$.

There is a partial order on the reduced diagrams. We say that a diagram $H$ is a prefix of a diagram $W$ and we write $H<W$ if $top(H)=top(W)$ and there exists a diagram $F$ with $top(F)=bot(H)$ such that $H\circ F=W$ (this includes the case $H=W$.) In particular $H\circ F$ is a reduced diagram, no cancelation is possible. Since $F$ is obtained from its top path $\overline F$ by attachments of consecutive cells hence $W$ is obtained from $H$ by attachments of consecutive cells, without cancelations. Diagram $F$ is called a {\em suffix} of $W$.

 We describe some structure on the cells in a diagram $W$.  The cells which are attached to the top path $\overline W$ are called the cells of the first generation. Any collection of these cells can be attached in an arbitrary order to the top path $\overline W$. When we attach all of them we get the diagram $D_1$. If $C_i$ is a cell of the first generation then $\underline C_{\,i}$ is contained in $\underline D_{\,1}$ and is disjoint from (has no common edge with) $\overline D_1=\overline W$. 
$D_1$ is a prefix of $W$. We consider next the cells which are attached to $\underline D_{\,1}$ from below. These are the cells of the second generation. When we attach all of these cells we get the diagram $D_2$. If $C_j$ is a cell of the second generation then $\underline C_{\,j}$ is contained in $\underline D_{\,2}$ and is disjoint from $\underline D_{\,1}$.  And so on. We now define a partial order on the cells in $W$. A cell $C_i$ of generation $i$ is earlier (a predecessor) of a cell $C_j$ of generation $i+1$ if $\underline C_{\,i}$ has a common edge with $\overline C_j$.
We write $C_i\prec C_j$. The transitive closure of this relation is a partial order on the cells of $W$. This partial order is also denoted by "$\prec$".
In particular an earlier cell must be of an earlier generation. Different cells of the same generation are incomparable.

It follows from this description that if $C_i\prec C_j$ then in every sequence of attachments of cells, without cancelation,  which start with $\overline W$ and ends with $W$ the cell $C_i$ must be attached before the cell $C_j$. If $H$ is a prefix of $W$ and if the cell $C_j$ is contained in $H$ then the cell $C_i$ is also contained in $H$.

A diagram $\Phi$ is called $thin$ if all of its cells are of the first generation - for every cell $C$ in $\Phi$ we have $\overline C$ contained in $\overline \Phi$. Then we may attach the cells of $\Phi$ to $\overline\Phi$ in an arbitrary order and when we attach any subfamily of the cells of $\Phi$ to $\overline\Phi$ we get a prefix of $\Phi$.\par\bigskip

 We fix a word $w\in S_A$ and consider the set
$\mathcal{V}(P,w)$ of all reduced diagrams $U$ with $top(U)=w$. They form the
vertices of an abstract cubical complex $K_{ab}(P,w)$ (see \cite{farley}). A $k$-dimensional cell of this complex is a $k$-dimensional cube $Q$ defined in the following way. We fix a vertex $U$ and a thin diagram $\Phi$ with $top(\Phi)=bot(U)$. The set of the vertices of $Q$ coincides with $C(U,\Phi)=\{UV:V<\Phi\}$. A prefix $\Phi_I$ of $\Phi$ is also a thin diagram and defines a face of $Q$. In particular a pair of vertices of $K_{ab}(P,w)$ form an edge if and only if one of the corresponding diagrams is obtained from the other by attaching one cell.

We consider the geometric realization of this complex, called $K(P,w)$, where each $k$-dimensional cube is given the standard metric of the unit $k$-dimensional cube. This induces the intrinsic length metric on the complex $K(P,w)$ (see Section 3) 
and makes it into a cubical complex.

 Cubical complexes were studied in \cite{gromov} and in \cite{BH}. Farley proved in \cite{farley}, Theorem 3.13 that $K(P,w)$ is a CAT(0) complex.

A diagram $U$ in $\mathcal{V}(P,w)$ is called {\em spherical} if $top(U)=bot(U)=w$. The reduced spherical diagrams form the diagram group  $G=D(P,w)$ with respect to the composition of diagrams. Let $g$ be a fixed spherical diagram in $D(P,w)$. Then $g$ acts on
the vertices of $K(P,w)$ by multiplication on the left $g(U)=gU$. This action extends to a linear map on each cube
and defines an isometry of the complex $K(P,w)$.

Farley has asked in \cite{farley}, Question 3.17, whether for every $P$ and $w$ the group $D(P,w)$ acts by semi-simple isometries and he suggested that the answer is affirmative.  The problem is known as  Farley Conjecture\par\smallskip
{\bf Farley Conjecture.} For every $P$ and $w$ the diagram group $D(P,w)$ acts on $K(P,w)$ by semisimple isometries.\par\smallskip

We shall prove

\begin{theorem} Farley conjecture is false for R. Thompson group $F$.\end{theorem}

Recall (see \cite{GS}, Example 6.4) that R. Thompson group $F$ is isomorphic to the diagram group $D(P,x)$ for the presentation of the trivial semigroup $P$ with one generator $x$ and one relation $x=x^2$. In the diagrams for $P$ there are only two kinds of cells: an {\em upper cell} $C(x,xx)$ and a {\em lower cell} $C(xx,x)$. We shall prove that the element $g$ on Figure \ref{diagramg}, left side, acts as a parabolic isometry on $K(P,x)$.

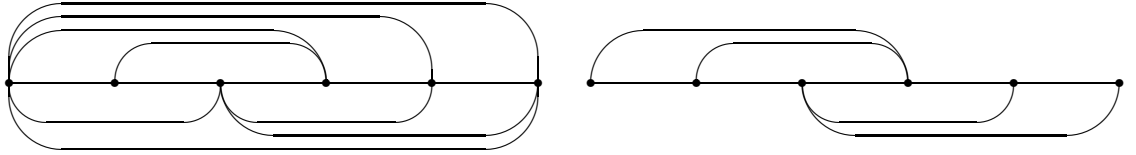
\begin{figure}
\centering
\begin{picture}(400,120)(20,-20)

\put(10,0){\line(1,0){200}}   
 \put(10,0){\circle*{3}}
 \put(50,0){\circle*{3}} \put(90,0){\circle*{3}}
 \put(130,0){\circle*{3}}
 \put(170,0){\circle*{3}}
 \put(210,0){\circle*{3}}

\put(90,0){\oval(80,30)[t]} \put(70,0){\oval(120,40)[t]}
 \put(130,0){\oval(80,30)[b]}  \put(150,0){\oval(120,40)[b]}
 \put(90,0){\oval(160,50)[t]}    \put(110,0){\oval(200,60)[t]}

 \put(50,0){\oval(80,30)[b]}   \put(110,0){\oval(200,50)[b]}

\put(230,0){\line(1,0){200}}   
 \put(230,0){\circle*{3}}
 \put(270,0){\circle*{3}} \put(310,0){\circle*{3}}
 \put(350,0){\circle*{3}}
 \put(390,0){\circle*{3}}
 \put(430,0){\circle*{3}}

\put(310,0){\oval(80,30)[t]} \put(290,0){\oval(120,40)[t]}
 \put(350,0){\oval(80,30)[b]}  \put(370,0){\oval(120,40)[b]}

\end{picture}
\caption{The diagram defining $g$.}\label{diagramg}
\end{figure}

\par\medskip
The following result is known to experts and is not difficult to prove
\begin{theorem}
\label{wypuklosc}
Let $K$ be a locally finite,  CAT(0) cubical complex. Let $K'$ be a finite-dimensional,
connected subcomplex of $K$. If
\newline
 $(*)$ for every cube $C$ in $K$ the intersection $C\cap K'$ is a cube or empty

then $K'$ is a convex subcomplex of $K$.
\end{theorem}
 \section{Convex subcomplexes of $K(P,w)$}

\begin{definition}
 Let $U,V$ be diagrams  in $\mathcal{V}(P,w)$ and let $W=U^{-1}V$.  Then
$K_{U,V}$ denotes the full subcomplex of $K(P,w)$ with vertices
$\{UH:H<W\}$. \end{definition}

One of the main tools in the proof of Theorem 1 is the following
\begin{theorem}\label{convex}
 Complex $K_{U,V}$ is a convex subcomplex of $K(P,w)$.\end{theorem}
\par\smallskip

We recall two results of Farley (see \cite{farley}, Proposition 2.2 and Lemma 2.4).

\begin{lemma} \label{anyvertex} \begin{enumerate} \item If $U,V$ are reduced diagrams with $top(U)=top(V)$ then there exists the greatest lower bound $W=glb(U,V)$ in the prefix order of diagrams.
\item  For any abstract cube  $C(U,\Phi)$ there is a reduced diagram $U_1$ (a vertex of $C(U,\Phi)$) and a thin diagram $\Psi$ such that $U_1\circ\Psi$ is reduced and  $C(U,\Phi)=C(U_1,\Psi)$. For every diagram $U_2\in C(U,\Phi)$  there exists a thin diagram $\Theta$ such that $C(U,\Phi)=C(U_2,\Theta)$.
\end{enumerate}
\end{lemma}

{\bf Proof of Theorem \ref{convex}.}
 According to Theorem \ref{wypuklosc} it is enough to show that $ K_{U,V}$ is connected and that it satisfies property $(*)$. Every vertex in  $ K_{U,V}$ is represented by a diagram, which is obtained from $U$ by a sequence of attachments of cells and thus it is connected to $U$ by a path of edges.

 Let $Q$ be a $k$-dimensonal cube of  $K(P,w)$. Suppose  the intersection  of $Q$ and $ K_{U,V}$ is not empty. For any vertex $Y\in Q\cap K_{U,V}$ there is a prefix $H<W$ such that $Y=UH$. We may choose a vertex $Y\in Q\cap K_{U,V}$ for which the prefix $H$ is minimal.  By Lemma \ref{anyvertex} there exists a thin diagram $\Phi$ such that
 $Q=C(Y,\Phi)$. We shall prove that $H\circ\Phi$ is reduced.
Suppose the thin diagram $\Phi$ contains a cell which cancels in $H\Phi$. Let $\Phi_1$ be the atomic prefix of $\Phi$ corresponding to this cell. Then $H_1=H\Phi_1<H<W$ and $Y\Phi_1=UH\Phi_1=UH_1$ is a vertex of $Q$ which belongs to $K_{U,V}$ and corresponds to the prefix $H_1$ smaller than $H$. Contradiction. Therefore $H\circ\Phi$ is reduced.

Since $H<W$ there exists a reduced diagram $L$ such that $H\circ L=W$. Since $top(\Phi)=top(L)=bot(H)$ there exists $\Phi_0=glb(L,\Phi)$. Then the face $Q_1=C(Y,\Phi_0)$ of $Q$  is contained in $K_{U,V}$. We shall prove that $Q\cap K_{U,V}=Q_1$. Suppose some vertex $Y\circ\Psi$ of $Q$ belongs to $K_{U,V}$, where $\Psi<\Phi$. Then $Y\Psi=UH_1$ for some $H_1<W$. Since $Y=UH$ we have $H\Psi=H_1<W$ and since $H\circ \Psi$ is reduced we have $\Psi<L$.
Thus $\Psi$ is a lower bound of $L$ and $\Phi$ and therefore it is a prefix of $\Phi_0$ and $Y\circ\Psi\in Q_1$. Therfore $Q\cap K_{U,V}=Q_1$ is a face of $Q$.
This completes the proof of Theorem \ref{convex}.

\section {The action of $g$ on a cube $Q$}
In order to compare distances inside $K_{U,V}$ we shall embed it into an Euclidean cube.

\begin{lemma}\label{embedding} There exists an isometric cellular embedding $\rho$ of subcomplex $K_{U,V}$ into the complex $I^N$  (an $N$-dimensional unit cube) where $N$ is the number of cells in $W=U^{-1}V$ and the metric in the image $\rho(K_{U,V})$ is the intrinsic metric induced in the subcomplex from the standard metric in the cubes of $\rho(K_{U,V})$, faces of $I^N$. The length of a given PL-path in $\rho(K_{U,V})$ and in $I^N$ is the same.
Vertices $U$ and $V$ are mapped onto the opposite vertices of the cube $I^N$ so their distance in $K(P,w)$ is at least $\sqrt{N}$.
\end{lemma}
\begin{proof}
We shall describe the embedding. We enumerate the cells in $W$: $C_1,C_2,\dots,C_N$.
We map a vertex  $UH$ of $K_{U,V}$ onto a vertex $\rho(UH)$ of $I^N$ with coordinates  $x_i$, $i=1,2,\dots,N$ where 
$x_i=1$ if  the cell $C_i$ is contained in $H$  and $x_i=0$ if $C_i$ is not contained in $H$. 

We check now that this is a cellular map. Suppose $Q$ is a $k$-dimensional cube in $K_{U,V}$.
As in the proof of Theorem 2 there is a prefix $H<W$ and a thin diagram $\Phi$ such that $H\circ \Phi$ is reduced and $Q=C(UH,\Phi)$ and $H\circ\Phi<W$. Then the vertices of $\rho(Q)$ are exactly those vertices of $I^N$ for which $x_i=1$ for all coordinates corresponding to the cells of $H$ and $x_j=0$ for all coordinates corresponding to cells which lie neither in $H$ nor in $\Phi$ and the other coordinates are either 0 or 1. Thus the vertices of $\rho(Q)$ coincides with the vertices of a face of $I^N$ and $\rho$ is a cellular map on the vertices. It extends linearly to the interior points of cubical cells of $K_{U,V}$. 

Since $K_{U,V}$ is convex, the metric from $K(P,w)$ restricted to  $K_{U,V}$ is the intrinsic metric induced by the metric in the cubes and the same is true for the metric in the image of $K_{U,V}$ in $I^N$.  Thus the embedding is an isometry.
\end{proof}

\subsection{Cubes in $K(P,w)$}

In this subsection $g$ is a fixed element of the group $D(P,w)$ and $Q=C(U,\Phi)$ is a fixed $k$-dimensional cube in $K(P,w)$. 

 Diagram $\Phi$ has $k$ cells ordered from left to right: $E_1,E_2,\dots,E_k$. Let $\mathcal{K}$ denote the set of integers $\{1,2,\dots,k\}$.
For each subset $J\subset \mathcal{K}$ we let $J'=\mathcal{K}-J$. We denote by $\Phi_J$ the prefix of $\Phi$ containing exactly the cells $E_i$ for $i\in J$. 
Every prefix of $\Phi$ has this form. Diagram $\Phi_{J'}^{-1}\Phi_J$ is a thin diagram which contains $k$ cells, the cells $E_i$ for $i\in J$ and cells $E_i'$ for $i\in J'$, where $E_i'$ is a cell of the type opposite to $E_i$: if $E_i$ is of type $C(u_i,v_i)$ then $E_i'$ is of type $C(v_i,u_i)$. 

 Every vertex of $Q$ has the form $U\Phi_J$. By Lemma \ref{anyvertex} there is a thin diagram $\Theta$ such that $Q=C(U\Phi_J,\Theta)$ and in fact diagram $\Theta=\Phi_J^{-1}\Phi_{J'}$ has this property.\par\smallskip

We shall find vertices $U_0$ of $Q$ and $V_0$ of $gQ=C(gU,\Phi)$ such that both cubes $Q$ and $gQ$ are contained in the complex $K_{U_0,V_0}$. When we embed the complex $K_{U_0,V_0}$ in the cube $I^N$, as in Lemma \ref{embedding}, then the vertices $U_0$ and $V_0$ correspond to the opposite vertices of the cube $I^N$ one with all the coordinates equal to 0 and the other with all the coordinates equal to 1.

We now describe a particular subset $T\subset \mathcal K$. Let $W_0=U^{-1}gU$.
We have $top(W_0)=bot(U)=top(\Phi)$ so there exists
$\Phi_T=glb(W_0,\Phi)$. Diagram $\Phi_T$ contains exactly those cells $E_i$ of $W_0$ (or  of $\Phi$) which cancel in the concatenation $\Phi^{-1}W_0$. Then the diagram $W_1=\Phi_{T'}^{-1}\circ W_0$ is reduced and contains prefix $\Phi_{T'}^{-1}\Phi_{T}$. In particular it contains cells $E_i$ for $i\in T$ (these cells belong also to $W_0$) and the new cells $E_i'$ for $i\in T'$ (these cells are attached to $W_0$ along $\overline{W_0}$ from above). We let $U_0=U\Phi_{T'}$. Since $\Phi_{T'}<\Phi$ diagram $U_0$ is a vertex of $Q$. 

We now consider the largest prefix $\Phi_L$ of $\Phi$ which cancels in $W_1\Phi$. It can be also defined as $\Phi_L=glb(W_1^{-1},\Phi)$. The cell in $W_1$ which cancels with $E_i$ in $W_1\Phi$ will be called $F_i'$. It is of the type opposite to the type of $E_i$. There are two possibilities for a cell $F_i'$.
It may be a cell which belongs to $W_0$. It may even be one of the cells $E_j$ in $W_0$ so we may have $E_j=F_i'$. Or it may be a new cell $E_j'$ in $W_1$ for some $j\in T'$ if $E_j'$ was attached to $W_0$ along a bridge - an arc of $\overline{W_0}$ which is also contained in  $\underline{W_0}$. We let $V_0=gU \Phi_{L'}$ and $W=U_0^{-1}V_0=W_1\Phi_{L'}$. For $i\in L'$ the cell $E_i$ of $\Phi_{L'}$ attached to $W_1$ and considered as a cell of $W$ will be denoted by $F_i$. It is of the same type as $E_i$.

All cells $E_i$ and $E_i'$ are attached to $\overline{W}$ from below and all cells $F_i$ and $F_i'$ are attached to $\underline{W}$ from above.

Diagram $W_1\circ \Phi_{L'}=W$ is reduced. Diagram $\Phi_L$ cancels completely in $W_1\Phi_L$ therefore $W_2=W_1\Phi_L$ is a prefix of $W_1$, we have just removed from $W_1$ all cells $F_j'$ for $j\in L$. Now $W=W_2\circ(\Phi_L^{-1}\Phi_{L'})$
is reduced and contains the suffix $\Phi_L^{-1}\Phi_{L'}$.

Observe that $W$ contains the prefix $\Phi_{T'}^{-1}\Phi_{T}$ and therefore $K_{U_0,V_0}$ contains the cube $Q=C(U,\Phi)=C(U\Phi_{T'},\Phi_{T'}^{-1}\Phi_{T})=C(U_0,\Phi_{T'}^{-1}\Phi_{T})$.

Also $W=W_2\circ (\Phi_L^{-1}\Phi_{L'})$ is reduced, therefore for any prefix $\Theta$ of $\Phi_L^{-1}\Phi_{L'}$ the diagram $W_2\circ\Theta$ is a prefix of $W$. We also have  $gU=U_0U_0^{-1}gU=U_0W_1=U_0W_2\Phi_L^{-1}$.

 It follows that $K_{U_0,V_0}$ contains the cube 
$gQ=C(gU,\Phi)=C(U_0W_2\Phi_L^{-1},\Phi)=C(U_0W_2,\Phi_L^{-1}\Phi_{L'})$.

We have proven

\begin{lemma}\label{diagram for Q} Complex $K_{U_0,V_0}$ described above contains the cube $Q$ and contains its image $gQ$.\end{lemma}

\subsection {Action of $g$ on a cube $Q$}
For a given cube $Q$ of dimension $k$ in $K(P,w)$ we may consider the subcomplex $K_{U_0,V_0}$ containing $Q$ and $gQ$, as in the previous lemma and we may consider the embedding of $K_{U_0,V_0}$ into the cube $I^N$. The action of $g$ restricted to $Q$ translates into the action on the coordinates in the cube $I^N$. Coordinates in $I^N$ correspond to cells $C_i$ contained in $W$. The coordinate corresponding to $C_i$ will be called $c_i$ so we shall say $c_i$-coordinate instead of $i$-th coordinate. We may enumerate the cells in such a way that for $i=1,\dots,k$ the coordinate $c_i$ corresponds to the cell $E_i$ if $i\in T$ and corresponds to the cell $E_i'$ if $i\in T'$. The next lemma describes the action of $g$ in $I^N$ restricted to $Q$. We use the notation from Subsection 4.1.

\begin{lemma}\label{actioncube} Let $Q$ be a $k$-dimensional cube in $K(P,w)$, let $K_{U_0,V_0}$ be the subcomplex of $K(P,w)$ containing $Q$ and $gQ$ as in Subsection 4.1. Consider the embedding of $K_{U_0,V_0}$ in $I^N$. Let $y$ be a point in $Q\subset I^N$ with coordinates $c_i=y_i\in[0,1]$ for $i=1,\dots,k$ and $c_i=0$ for $i>k$. We consider the induced action of $g$ in $Q\subset I^N$. Then $z=g(y)$ has coordinates $c_i=z_i$ where\begin{enumerate}
\item $z_i=1-y_j$ if $C_i=F_j'$ for some $j\in L\cap T$
\item $z_i=0+y_j$ if $C_i=F_j'$ for some $j\in L\cap T'$
\item  $z_i=1$ if $C_i$ is not equal to any $F_j$ or $F_j'$.
\item  $z_i=0+y_j$ if $C_i=F_j$ for some $j\in L'\cap T$
\item $z_i=1-y_j$ if $C_i=F_j$ for some $j\in L'\cap T'$

\end{enumerate}
\end{lemma}
\begin{proof} Let $P_0$ denote the origin in $I^N$ (the point with all coordinates equal to 0) and let $P_i$ denote the vertex of $I^N$ corresponding to the cell $C_i$ (the coordinate  $c_i$).
The action of $g$ in $I^N$ is affine, it does not preserve the origin, so we describe a point $y$ as a combination 
$$y=P_0+\Sigma_{j=1}^k y_j(P_j-P_0).$$
Then
$$gy=gP_0+\Sigma_{j=1}^k y_j(gP_j-gP_0).$$

Each cell $C_i$ of $W$  corresponds to exactly one of the cases (1) - (5) in the Lemma.

We compute first the coordinates of the point $gP_0$. It corresponds to the diagram $gU_0=U_0U_0^{-1}gU_0=U_0W_1\Phi_{T'}$ and to the prefix $W_1\Phi_{T'}$ of $W$. We start with all the cells of $W_1$. These are the cells $C_i$ not equal to any $F_j$ for $j\in L'$, cases (1), (2) and (3). Next we attach the cells of $\Phi_{T'}$. We
 attach the cells $F_j$ for $j\in L'\cap T'$ (case (5))  and we remove the cells $F_j'$ for $j\in L\cap T'$ (case (2)). The cells in $W_1\Phi_{T'}$ correspond to coordinates of $gP_0$ which are equal 1. The other coordinates of $gP_0$ are equal 0. So the coordinate $c_i$ of $gP_0$ is equal 1 in cases (1), (3), (5) and is equal 0 in cases (2) and (4).

We consider now $gP_j$ for $j=1,2,\dots,k$. Vertex $P_j$ corresponds to the diagram obtained from $U_0$ by the attachment of one cell of $\Phi_{T'}^{-1}\Phi_T$. Therefore $gP_j$ corresponds to the diagram obtained from $gU_0$ in the same way. The difference $gP_j-gP_0$ has only one coordinate
different from 0 which corresponds to the attached cell. We start with $W_1\Phi_{T'}$ and attach one cell: $F_j$ if $j\in T$ and $F_j'$ if $j\in T'$.

If $j\in L'\cap T$ we attach the cell $F_j$ and for $C_i=F_j$ the coordinate $c_i$ in $gy$ is increased by $y_j$. We are in case (4).

If $j\in L\cap T$ we attach the cell $F_j$. The cell $F_j'$ was contained in $W_1$, so "we now remove a part of it". For $C_i=F_j'$ the coordinate $c_i$ in $gy$ is decreased by $y_j$. We are in case (1).

 If $j\in L'\cap T'$ we attach the cell $F_j'$. The cell $F_j$ was attached earlier to $W_1$ in the construction of $W_1\Phi_{T'}$, so "we now remove a part of it". For $C_i=F_j$ the coordinate $c_i$ in $gy$ decreases by $y_j$. We are in case (5).

If $j\in L\cap T'$ we attach the cell $F_j'$. Since $j\in L$ this cell existed in $W_1$, it was canceled when we attached $\Phi_{T'}$ and it is now added again. For $C_i=F_j'$ the coordinate $c_i$ in $gy$ increases by $y_j$. We are in case (2).
\end{proof}

\section{Considerations in R. Thompson's group $F$}

We restrict now our attention to the R. Thompson's group $F$ represented as a diagram group $D(P,x)$ with $P=<x|x=xx>$. 
 In the case of the semi-group $P$ the group $D(P,w)$ is independent of the word $w$. We shall prove it for a particular word $w=xxx$ which we shall use. The diagrams in $K(P,x)$ have the top path with one edge labeled $x$. If we attach a cell of type $C(xx,x)$ from above and then again a cell $C(xx,x)$ from above along the first edge of the new top path we get a diagram with top word $xxx$. This defines an isometry of $K(P,x)$ onto $K(P,w)$. If we attach the two cells from above to the diagrams in $D(P,x)$ and in a symmetric way two cells from below we get an isomorphism of the group $D(P,x)$ onto the group $D(P,w)$ which is equivariant with respect to the group actions on the respective cube complexes. The element $g$ represented by the diagram on the left side of Figure \ref{diagramg} maps onto the element represented on the right side, which looks simpler and contains only 4 cells. Because of this we shall consider further only the group $D(P,w)$, where $w=xxx$, and we shall denote by $g$ the element represented by the diagram on the right side of Figure \ref{diagramg}.

\subsection{Bounding $|g|$ from below}
\begin{lemma} If $U$ is a diagram in $K(P,w)$ then $W=U^{-1}gU$ has at least 4 cells.
\end{lemma}
\begin{proof} We shall prove that if $W=U^{-1}gU$ has at most 4 cells then it has exactly 4 cells and it has one of the forms on Figure \ref{4cells}.

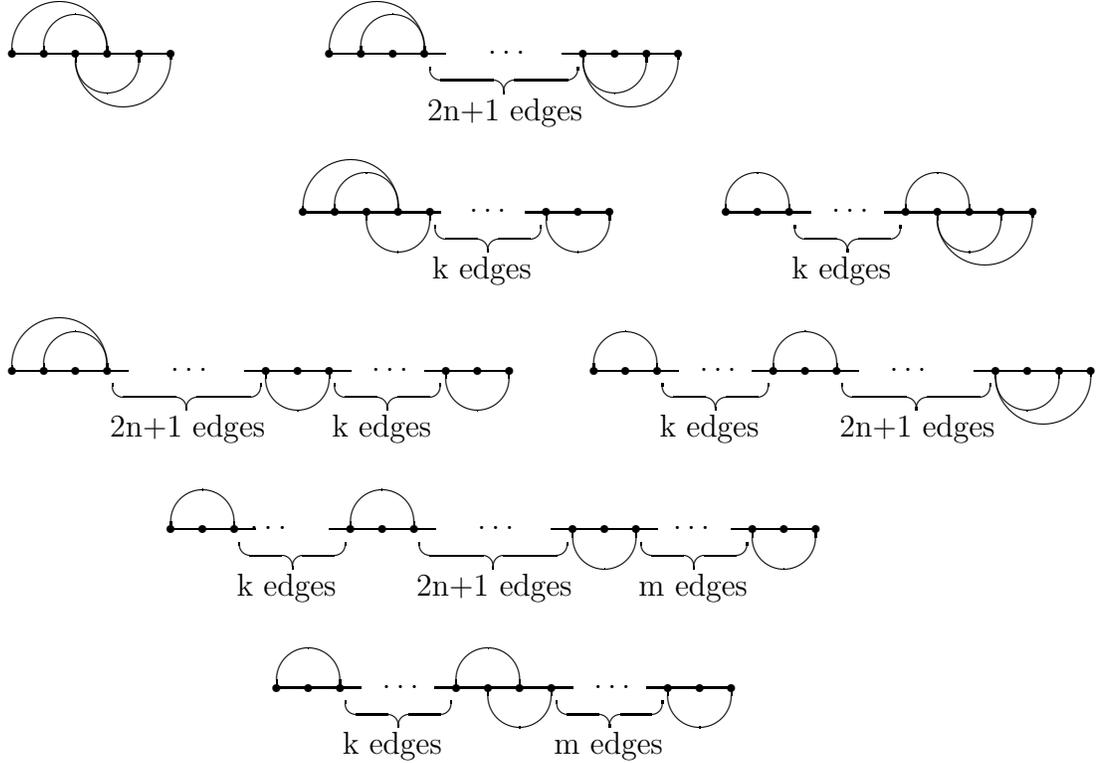
\begin{figure}
\centering
\begin{picture}(400,400)(20,-180)

\put(10,120){\line(1,0){60}}   
 \put(10,120){\circle*{3}}  \put(22,120){\circle*{3}}
 \put(34,120){\circle*{3}}  \put(46,120){\circle*{3}}
 \put(58,120){\circle*{3}}  \put(70,120){\circle*{3}}

\put(28,120){\oval(36,40)[t]} \put(34,120){\oval(24,30)[t]}

 \put(46,120){\oval(24,30)[b]}   
\put(52,120){\oval(36,40)[b]}

\put(130,120){\line(1,0){44}}   \put(218,120){\line(1,0){44}} 
 \put(130,120){\circle*{3}}  \put(142,120){\circle*{3}}
 \put(154,120){\circle*{3}}  \put(166,120){\circle*{3}}
 \put(226,120){\circle*{3}}  \put(238,120){\circle*{3}}
 \put(250,120){\circle*{3}}  \put(262,120){\circle*{3}}

\put(148,120){\oval(36,40)[t]} \put(154,120){\oval(24,30)[t]}

 \put(238,120){\oval(24,30)[b]}   
\put(244,120){\oval(36,40)[b]}
 \put(190,120){$\dots$} 

\put(182,115){\oval(28,10)[lb]}  \put(182,105){\oval(28,10)[rt]}

\put(210,105){\oval(28,10)[lt]} \put(210,115){\oval(28,10)[rb]} 

\put(167,95){2n+1 edges}

\put(120,60){\line(1,0){52}}   \put(204,60){\line(1,0){32}}   
 \put(120,60){\circle*{3}}  \put(132,60){\circle*{3}}
 \put(144,60){\circle*{3}}  \put(156,60){\circle*{3}}
 \put(168,60){\circle*{3}}  \put(212,60){\circle*{3}}
 \put(224,60){\circle*{3}}  \put(236,60){\circle*{3}}
 \put(183,60){$\dots$} 

\put(138,60){\oval(36,40)[t]} \put(144,60){\oval(24,30)[t]}

 \put(156,60){\oval(24,30)[b]}
\put(224,60){\oval(24,30)[b]}
\put(180,55){\oval(20,10)[lb]}  \put(180,45){\oval(20,10)[rt]}

\put(200,45){\oval(20,10)[lt]} \put(200,55){\oval(20,10)[rb]} 

\put(169,35){k edges}

\put(280,60){\line(1,0){32}}   \put(340,60){\line(1,0){56}}   
 \put(280,60){\circle*{3}}  \put(292,60){\circle*{3}}
 \put(304,60){\circle*{3}}  \put(348,60){\circle*{3}}
 \put(360,60){\circle*{3}}  \put(372,60){\circle*{3}}
 \put(384,60){\circle*{3}}  \put(396,60){\circle*{3}}
 \put(320,60){$\dots$} 

\put(292,60){\oval(24,30)[t]} \put(360,60){\oval(24,30)[t]}

 \put(372,60){\oval(24,30)[b]}
\put(378,60){\oval(36,40)[b]}
\put(316,55){\oval(20,10)[lb]}  \put(316,45){\oval(20,10)[rt]}

\put(336,45){\oval(20,10)[lt]} \put(336,55){\oval(20,10)[rb]} 

\put(305,35){k edges}

\put(10,0){\line(1,0){44}}   \put(98,0){\line(1,0){40}}   
\put(166,0){\line(1,0){32}}   
 \put(10,0){\circle*{3}}  \put(22,0){\circle*{3}}
 \put(34,0){\circle*{3}}  \put(46,0){\circle*{3}}
 \put(106,0){\circle*{3}}  \put(118,0){\circle*{3}}
 \put(130,0){\circle*{3}}  \put(174,0){\circle*{3}}
 \put(186,0){\circle*{3}}  \put(198,0){\circle*{3}}

\put(28,0){\oval(36,40)[t]} \put(34,0){\oval(24,30)[t]}

 \put(118,0){\oval(24,30)[b]}
\put(186,0){\oval(24,30)[b]}

 \put(70,0){$\dots$}   \put(146,0){$\dots$} 
\put(62,-5){\oval(28,10)[lb]}  \put(62,-15){\oval(28,10)[rt]}

\put(90,-15){\oval(28,10)[lt]} \put(90,-5){\oval(28,10)[rb]} 

\put(47,-25){2n+1 edges}

\put(142,-5){\oval(20,10)[lb]}  \put(142,-15){\oval(20,10)[rt]}

\put(162,-15){\oval(20,10)[lt]} \put(162,-5){\oval(20,10)[rb]} 

\put(131,-25){k edges}

\put(230,0){\line(1,0){32}}   \put(290,0){\line(1,0){40}}   
\put(374,0){\line(1,0){44}}   
 \put(230,0){\circle*{3}}  \put(242,0){\circle*{3}}
 \put(254,0){\circle*{3}}  \put(298,0){\circle*{3}}
 \put(310,0){\circle*{3}}  \put(322,0){\circle*{3}}
 \put(382,0){\circle*{3}}  \put(394,0){\circle*{3}}
 \put(406,0){\circle*{3}}  \put(418,0){\circle*{3}}

\put(242,0){\oval(24,30)[t]} \put(310,0){\oval(24,30)[t]}

 \put(400,0){\oval(36,40)[b]}
\put(394,0){\oval(24,30)[b]}

 \put(270,0){$\dots$}   \put(342,0){$\dots$} 
\put(338,-5){\oval(28,10)[lb]}  \put(338,-15){\oval(28,10)[rt]}

\put(366,-15){\oval(28,10)[lt]} \put(366,-5){\oval(28,10)[rb]} 

\put(323,-25){2n+1 edges}

\put(266,-5){\oval(20,10)[lb]}  \put(266,-15){\oval(20,10)[rt]}

\put(286,-15){\oval(20,10)[lt]} \put(286,-5){\oval(20,10)[rb]} 

\put(255,-25){k edges}

\put(110,-120){\line(1,0){32}}   \put(170,-120){\line(1,0){52}} 
    \put(250,-120){\line(1,0){32}}   
 \put(110,-120){\circle*{3}}  \put(122,-120){\circle*{3}}
 \put(134,-120){\circle*{3}}  \put(178,-120){\circle*{3}}
 \put(190,-120){\circle*{3}}  \put(202,-120){\circle*{3}}
 \put(214,-120){\circle*{3}}  \put(258,-120){\circle*{3}}
 \put(270,-120){\circle*{3}}  \put(282,-120){\circle*{3}}
 \put(150,-120){$\dots$}    \put(230,-120){$\dots$} 

\put(122,-120){\oval(24,30)[t]} \put(190,-120){\oval(24,30)[t]}

 \put(202,-120){\oval(24,30)[b]}
\put(270,-120){\oval(24,30)[b]}
\put(146,-125){\oval(20,10)[lb]}  \put(146,-135){\oval(20,10)[rt]}

\put(166,-135){\oval(20,10)[lt]} \put(166,-125){\oval(20,10)[rb]} 

\put(135,-145){k edges}

\put(226,-125){\oval(20,10)[lb]}  \put(226,-135){\oval(20,10)[rt]}

\put(246,-135){\oval(20,10)[lt]} \put(246,-125){\oval(20,10)[rb]} 

\put(215,-145){m edges}

\put(70,-60){\line(1,0){32}}   \put(130,-60){\line(1,0){40}}   
\put(214,-60){\line(1,0){40}}    \put(282,-60){\line(1,0){32}}   
 \put(70,-60){\circle*{3}}  \put(82,-60){\circle*{3}}
 \put(94,-60){\circle*{3}}  \put(138,-60){\circle*{3}}
 \put(162,-60){\circle*{3}}  \put(222,-60){\circle*{3}}
 \put(150,-60){\circle*{3}}  \put(234,-60){\circle*{3}}
 \put(246,-60){\circle*{3}}  \put(290,-60){\circle*{3}}
 \put(302,-60){\circle*{3}}  \put(314,-60){\circle*{3}}

\put(82,-60){\oval(24,30)[t]} \put(150,-60){\oval(24,30)[t]}

 \put(234,-60){\oval(24,30)[b]}
\put(302,-60){\oval(24,30)[b]}

 \put(100,-60){$\dots$}   \put(186,-60){$\dots$}   \put(260,-60){$\dots$}

\put(106,-65){\oval(20,10)[lb]}  \put(106,-75){\oval(20,10)[rt]}

\put(126,-75){\oval(20,10)[lt]} \put(126,-65){\oval(20,10)[rb]} 

\put(95,-85){k edges}

\put(178,-65){\oval(28,10)[lb]}  \put(178,-75){\oval(28,10)[rt]}

\put(206,-75){\oval(28,10)[lt]} \put(206,-65){\oval(28,10)[rb]} 

\put(163,-85){2n+1 edges}

\put(258,-65){\oval(20,10)[lb]}  \put(258,-75){\oval(20,10)[rt]}

\put(278,-75){\oval(20,10)[lt]} \put(278,-65){\oval(20,10)[rb]} 

\put(247,-85){m edges}

\end{picture}
\caption{Possible conjugates of $g$ with 4 cells.}\label{4cells}
\end{figure}

Suppose that we have one of the diagrams on Figure \ref{4cells} and we attach a cell from above along the top path of the diagram and a cell of the opposite type from below, along the corresponding edges of the bottom path of the diagram  in such a way that one of the cells cancels. There are at most eight possibilities in each of the  diagrams. It is easy to check that in each case we get again one of the 8 types of diagrams, up to an isotopy. Now let $W$ be any of the diagrams on Figure \ref{4cells}. Let $U$ be a diagram with $top(U)=bot(W)$. We claim that if $U^{-1}WU$ has at most 4 cells then it has exactly 4 cells and has one of the forms on Figure \ref{4cells}. We argue by induction on the number of cells in $U$. We have just seen that it is true when $U$ has only one cell. In general if $U^{-1}WU$ has at most 4 cells then either we are left with the diagram $W$ or one of the cells, call it $C_1$, of $W$ cancels in $U^{-1}WU$. We may assume by symmetry that eventually  $C_1$ cancels with a cell below it. There must be a cell $C_2$ of $U$ such that $\overline C_2$ and $\underline C_1$ have an edge in common. If $C_1$ and $C_2$ cancel then $C_2$ is a cell of the first generation in $U$. Then $U=C_2\circ U_1$. More precisely we have an atomic diagram $V$ with one cell $C_2$ and with $top(V)=bot(W)$, and $U=V\circ U_1$. We may start with $V^{-1}WV$, get some $W_1$ as on Figure \ref{4cells} and then $U_1^{-1}W_1U_1$ is again of such form by induction on the number of cells in $U$. If $C_2$ does not cancel with $C_1$ but $C_1$ cancels eventually then there must be a cell $C_3$ in $U$ which touches $C_2$ from below (has a common edge with $C_2$) and so on. These cells may not cancel between themselves because we may assume that $U$ is reduced. But we cannot have an infinite sequence of cells in $U$. Therefore $C_1$ must cancel with $C_2$ and the claim follows by induction. This proves the lemma.
\end{proof}
\begin{proposition}\label{distance} If $y$ is a point in $K(P,w)$ then  the distance $d(y,gy)$ in $K(P,w)$ is bigger than $\sqrt{2}$. If $y$ is a vertex then $d(y,gy)$ is at least 2 and if $y$ is the midpoint of a cube then $d(y,gy)$ is at least $\sqrt{2.25}$.
\end{proposition}

 If $y=U$ is a vertex in  $K(P,w)$ then $W=U^{-1}gU$ has at least 4 cells. We consider the subcomplex $K_{U,gU}$ and its embedding into a cube $I^N$. Then $I^N$ contains $U$ and $gU$ as opposite vertices and since $N\ge 4$ the distance  $d(y,gy)$ is at least 2. So we may assume that $y$ is in the interior of some cube $Q=C(U,\Phi)$ of a positive dimension $k>0$. We can choose $U_0,V_0$ as in Subsection   4.1 so that $K_{U_0,V_0}$ contains $Q$ and $gQ$. We use the notation from Subsection 4.2.  We can consider points $y$ and $gy$ in the cube $I^N$ and estimate the length of a path $\alpha$ connecting $y$ to $gy$ in $K_{U_0,V_0}\subset I^N$. Since geodesics in $K_{U_0,V_0}$ are PL-paths we may assume that $\alpha$ is a PL-path (so its length in $K_{U_0,V_0}$ and in $I^N$ is the same).
The coordinates of $y$ will be called $y_i$ and the coordinates of $gy$ will be called $z_i$. We shall choose a particular face $F$ of $I^N$ of dimension at most 4 and we shall prove that the projection of $\alpha$ onto $F$ has the length at least $\sqrt{2}$. Then we shall prove that at least one coordinate of $\alpha$ which is constant in $F$ is not constant in $\alpha$, and therefore the length of $\alpha$ is more than $\sqrt{2}$. For this we shall need two special cases of Lemma \ref{actioncube}.

\begin{lemma}\label{dependentcoordinate}

\begin{enumerate}

\item If  $i\in T\cap L'$ and $E_i\prec F_i$  (or  $i\in T'\cap L$ and $E_i'\prec F_i'$)  and if $F_i=C_p$ (respectively $F_i'=C_p$) then $y_i=z_p$ and $y_p=0$ and $z_i=1$.

\item If $p>k$  or if $C_p$ is not equal to any $F_i'$  then $y_p\ne z_p$.

\end{enumerate}
\end{lemma}
\begin{proof} If $E_i\prec F_i$ then $E_i$ is not attached to $\underline{W}$  and is not equal to any $F_j$ or $F_j'$ and $F_i$ is not attached to $\overline{W}$ and is not equal to any $E_j$ or $E_j'$. Therefore $p>k$ and $y_p=0$ and $z_i=1$ by case (3) of Lemma \ref{actioncube}. The case $E_i'\prec F_i'$ is similar. Consider $z_p$. We have $p>k$ and $C_p=F_i$ and $i\in T\cap L'$, hence $z_p=y_i$ by case (4) of Lemma \ref{actioncube} (respectively $C_p=F_i'$ and $i\in T'\cap L$, hence $z_p=y_i$ by case (2) of Lemma  \ref{actioncube}).

Consider statement (2). If $p>k$ then $y_p=0$ and $z_p\ne 0$ for all cases (1) - (5) of Lemma \ref{actioncube} because $y_j\in(0,1)$. If $p\le k$ and $C_p$ is not equal to any $F_i'$ then $y_p\in (0,1)$ and $z_p=1$ by case (3) of Lemma \ref{actioncube}.
\end{proof}

\par\bigskip
\begin{proof} (of Proposition \ref{distance}.)

Let $\alpha(t)$ be a path in $K_{U_0,V_0}\subset I^N$ connecting $y$ and $gy$. We shall distinguish 2, 3 or 4 coordinates of $I^N$ and we shall project the path $\alpha(t)$ onto these coordinates. We shall prove that the length of the projection is at least $\sqrt{2}$. We shall prove that at least one other coordinate is not constant in $\alpha(t)$ and therefore the length of $\alpha(t)$ is greater than $\sqrt{2}$.

Recall that $W_0=U^{-1}gU$ has at least four cells. Suppose first that $\overline W_0$ has only one edge. Then $Q$ has dimension 1,
there is a top cell $C_1=E_1$ in $W_0$, a bottom cell $C_2=F_1'$ and at least two middle cells $C_3$ and $C_4$. Point $y$ has $c_1$ coordinate less than 1 and coordinates $c_2,c_3,c_4$ equal 0.
Point $gy$ has the coordinates $c_1,c_2,c_3$ equal 1. Then the distance from $x$ to $gx$ in $I^N$ is bigger than $\sqrt{2}$ and the length of $\alpha(t)$ is at least the same.

We may assume that $\overline W_0$ has more than one edge , so there is no cell in $W_0$ which touches both the left end of $W_0$ and the right end of $W_0$.
A cell $C_p$ of $W$ will be called a {\em neutral} cell if it is not equal to any $E_i,E_i',F_i$ or $F_i'$.
There is an odd number of cells in $W_0$ which meet the left end of $W_0$, because we have started with one such cell in $g$ and then we have added symmetrically on top and bottom.
 If there are more than two cells of $W_0$ which meet the left end then one of them is neutral, it is not attached to $\overline W$ nor to $\underline W$. Whenever there exists a neutral cell touching the left end of $W_0$  we choose one of these cells, say $C_p$. We distinguish the coordinate $c_p$.

Suppose there is only one cell $C$ in $W_0$ which meets the left end and it is not a neutral cell. It must be an upper cell (of type $C(x,xx)$) because $g$ has a unique cell meeting the left end and it is an upper cell, so any conjugate of $g$ has more upper cells than lower cells meeting the left end. Since $C$ belongs to $W_0$ there are two possibilities $C=E_1$ or $C=F_1'$ (cells $E_i'$ and $F_j$ do not belong to $W_0$).

 Suppose $C=E_1$. Since it is the unique cell meeting the left end of $W_0$ it is leftmost cell along the $\overline W_0$ and along the $\underline W_0$. The cell $F_1$ is attached along the left lower edge of $E_1$ and does not cancel with $E_1$. Therefore $1\notin L$ and cell $F_1$ belongs to $W$. It corresponds to some coordinate $c_p$ so $F_1=C_p$. We distinguish coordinates $c_1$ and $c_p$. Observe that $C_1\prec C_p$ in the sense of the order of cells described in Section 2.

Suppose now that $C=F_1'$. It is of the same type as $E_1'$. 
The top edge of $C$ is the leftmost edge of $\overline W_0$. When we attach $E_1'$ along $\overline W_0$ from above it does not cancel. Therefore $1\in T'$, cell $E_1'$ belongs to $W_1\subset W$ and is equal to $C_1$. Cell $C$ corresponds to some coordinate $p>k$ and $F_1'=C_p$.
We distinguish both coordinates $c_1$ and $c_p$. Observe that $C_1\prec C_p$.

We now consider the cells meeting the right end of $W_0$ and we proceed in the same way. If there is a neutral cell $C_q$ in $W_0$ which meets the right end we distinguish the coordinate $c_q$.

If there is only one cell of $W_0$ meeting the right end of $W_0$ then it is a lower cell, of type $C(xx,x)$. If it is not a neutral cell then it is equal to  $E_k=C_k$ (or to $F_k'=C_q$).  Then the cell $F_k=C_q$ (respectively $E_k'=C_k$) also belongs to $W$ but does not belong to $W_0$ and $C_k\prec C_q$. We distinguish the coordinates $c_k$ and $c_q$.

We now estimate the length of the projection of $\alpha(t)$ onto the face of $I^N$ spanned by the distinguished coordinates. We start with the most complicated case when we distinguish 4 coordinates $c_1,c_p,c_k,c_q$. Since the cells $C_1$ and $C_k$ belong to a thin diagram $\Phi_{T'}^{-1}\Phi_T$  they are unrelated in the partial order of the cells in $W$ and the cells $C_p$ and $C_q$ likewise. On the other hand $C_1\prec C_p$ and $C_k\prec C_q$. Coordinates corresponding to any pair of unrelated cells form a square, a 2-dimensional cube, in the image of the subcomplex $K_{U_0,V_0}$ in $I^N$. No three of these coordinates form a cube in the image of $K_{U_0,V_0}$ in $I^N$. If there are no other relations we have 4 squares as on Figure \ref{projection}.

 The coordinates on Figure  \ref{projection}  are oriented from left to right and from top to bottom. We explain the meaning of this picture. The left upper square has the coordinates $c_p=c_q=0$ and  $c_1,c_k\in [0,1]$. The right upper square has the coordinates $c_p=0$, $c_k=1$ and $c_1,c_q\in [0,1]$.
The left lower square has the coordinates $c_1=1$, $c_q=0$ and $c_p,c_k\in [0,1]$.
The right lower square has the coordinates $c_1=1$, $c_k=1$ and $c_p,c_q\in [0,1]$.

The previous discussion means the following. If a point $z$ (a point $\alpha(t)$ in particular) belongs to the image of $K_{U_0,V_0}$ in $I^N$ then its projection onto the face spanned by the distinguished coordinates lies in the union of the 4 squares in Figure \ref{projection}. Indeed if $c_1<1$ then $c_p=0$ because $C_1\prec C_p$.
If $c_k<1$ then $c_q=0$ because $C_k\prec C_q$. Therefore $z$ projects into one of the four squares.

 The only other relations may be $C_1\prec C_q$, and then the upper right square is missing, or $C_k\prec C_p$, and then the lower left square is missing, or both these squares are missing (which does not seem to be possible). By Lemma \ref{dependentcoordinate} point $y$ has the coordinates $c_1=y_1$, $c_p=0$, $c_k=y_k$ and $c_q=0$ so it lies in the left upper square. Point $gy$ has the coordinates
$c_1=1$, $c_p=y_1$, $c_k=1$ and $c_q=y_k$ so it lies in the right lower square
in the same position. 

 The distance from $y$ to $gy$ in the union of squares is $\sqrt{2}$ and if one of the squares is missing the distance may be bigger. 

Suppose now that one of the cells of $W_0$ meeting the left end of $W_0$ is a neutral cell. Then it is called $C_p$ and the coordinate $c_1$ is missing. We only have the two lower squares on Figure \ref{projection}. 
Point $y$ has the coordinates $c_p=0$, $c_k=y_k$ and $c_q=0$ so it lies on the left upper edge of the rectangle. The point $gy$ has the coordinates $c_p=1$, $c_k=1$ and $c_q=y_k$ so it lies on the right lower edge of the rectangle in the same position. The distance from $y$ to $gy$ in the plane is $\sqrt{2}$ and if the left square is missing the distance may be bigger. 

If  a cell of $W_0$ meeting the right end of $W$ is a neutral cell we have a symmetric argument and if each end has such a cell then we project onto a square formed by 2 coordinates and the points $y$ and $gy$ are at the opposite corners, again in a distance $\sqrt{2}$.

We now prove that there is at least one other cell $C_i$ in $W$ such that $i>k$ or $C_i$ is not equal to any $F_j'$. Suppose it is not true. Then for every cell $C_i$ different from the distinguished cells we have $i\le k$ and $C_i=F_j'$ and at least two of these cells belong to $W_0$. Either both cells $C_1$ and $F_1$ (or $F_1'$) are distinguished or neither one of them is. Also either both cells $C_k$ and $F_k$ (or $F_k'$) are distinguished or none of them is. The non-distinguished cells consist of (possibly $C_1$), $C_2,\dots,C_{k-1}$,(possibly $C_k$) and the same sequence also consists of (possibly $F_1'$), $F_2',\dots,F_{k-1}'$, (possibly $F_k'$). Because of the ordering of the cells from left to right we must have $C_2=F_2'$, $C_3=F_3',\dots,C_{k-1}=F_{k-1}'$ and possibly $C_1=F_1'$ or $C_k=F_k'$ if they are not distinguished. But there exists some $C_i$ which belongs to $W_0$ and then $C_i=E_i=F_i'$. This is impossible because $F_i'$ is of type $E_i'$ and not of type $E_i$.

 This proves that there exists a non-distinguished coordinate $C_i$ such that either $i>k$ or $C_i\ne F_j'$  for any $j$. By Lemma \ref{dependentcoordinate}
$y_i\ne z_i$ so the $c_i$ coordinate is not constant along the path $\alpha(t)$ and the length of the path is greater than $\sqrt{2}$. If $y$ is the midpoint of $Q$ then $y_i=0.5$ and $z_i$ is equal 0 or 1, therefore $|y_i-z_i|=0.5$ and the length of $\alpha$ is at least $\sqrt{2.25}$.

This concludes the proof of Proposition \ref{distance}. \end{proof}

\begin{figure}
\centering
\begin{picture}(250,140)(20,-50)
  \put(10,50){\line(1,0){100}}   
   \put(10,0){\line(1,0){100}}   
   \put(10,-50){\line(1,0){100}}   
   \put(10,50){\line(0,-1){100}}   
   \put(60,50){\line(0,-1){100}}   
     \put(110,50){\line(0,-1){100}}   
 \put(10,50){\circle*{3}} \put(60,50){\circle*{3}}  \put(110,50){\circle*{3}}
 \put(10,0){\circle*{3}} \put(60,0){\circle*{3}}  \put(110,0){\circle*{3}}
 \put(10,-50){\circle*{3}} \put(60,-50){\circle*{3}}  \put(110,-50){\circle*{3}}

\put(-5,20){$c_1$}  \put(-5,-30){$c_p$}  \put(30,55){$c_k$}  \put(80,55){$c_q$}

\end{picture}
\caption{Projection onto 4 coordinates.}\label{projection}
\end{figure}
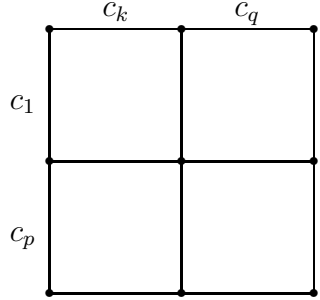
\par\bigskip
\subsection{Bounding $|g|$ from above}
\begin{lemma} For every $n$ there exists a diagram $U\in K(P,w)$ for which $W_0=U^{-1}gU$ has the form as on Figure \ref{example}.

\end{lemma}
\begin{proof} We shall attach consecutive cells from below starting with the diagram  $w$ (the degenerate diagram) in order to get the diagram $U$. We shall not look at the consecutive forms of $U$ but rather at the effect of the attachment of cells symmetrically on top and bottom starting from $g$. 

For a spherical diagram $V$ we denote by $L_i(V)$ the operation of the attachment of a cell $C(xx,x)$ from above along the $i$-th  edge of $\overline V$ and the attachment of a cell $C(x,xx)$ from below along the $i$-th edge of $\underline V$.  We perform the operations in order from right to left.

We start with $g$ and apply 

 $L_1^{n+1}L_{2n+3}L_{2n+1}L_{2n-1}\dots L_5L_3L_{n+4}L_{n+3}L_{n+2}L_{n+1}\dots L_5L_4L_3$. We get the diagram on Figure \ref{example}.

\end{proof}

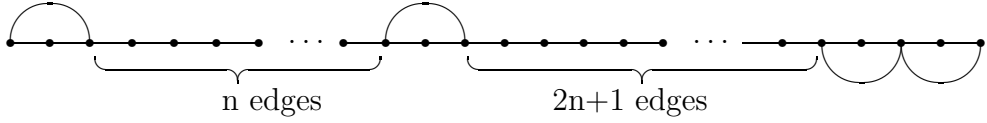
\begin{figure}
\centering
\begin{picture}(390,80)(20,-40)

\put(10,0){\line(1,0){94}}   \put(136,0){\line(1,0){121}}   
\put(287,0){\line(1,0){90}}   
 \put(10,0){\circle*{3}}  \put(25,0){\circle*{3}}
 \put(40,0){\circle*{3}}  \put(56,0){\circle*{3}}
 \put(72,0){\circle*{3}}  \put(88,0){\circle*{3}}
 \put(104,0){\circle*{3}}  \put(136,0){\circle*{3}}
 \put(152,0){\circle*{3}}  \put(167,0){\circle*{3}}
 \put(182,0){\circle*{3}}  \put(197,0){\circle*{3}}
 \put(212,0){\circle*{3}}  \put(227,0){\circle*{3}}  \put(242,0){\circle*{3}}
 \put(257,0){\circle*{3}} \put(302,0){\circle*{3}}
 \put(317,0){\circle*{3}}   \put(332,0){\circle*{3}}  \put(347,0){\circle*{3}}
 \put(362,0){\circle*{3}}  \put(377,0){\circle*{3}}

\put(25,0){\oval(30,30)[t]} \put(167,0){\oval(30,30)[t]}

 \put(332,0){\oval(30,30)[b]}
\put(362,0){\oval(30,30)[b]}

 \put(115,0){$\dots$}  
\put(69,-5){\oval(54,10)[lb]}  \put(69,-15){\oval(54,10)[rt]}

\put(123,-15){\oval(54,10)[lt]} \put(123,-5){\oval(54,10)[rb]} 

\put(90,-25){n edges}
\put(268,0){$\dots$}

\put(216,-5){\oval(66,10)[lb]}  \put(216,-15){\oval(66,10)[rt]}

\put(282,-15){\oval(66,10)[lt]} \put(282,-5){\oval(66,10)[rb]} 

\put(215,-25){2n+1 edges}
\end{picture}
\caption{The diagram $W_0$ for a suitable choice of $U$.}\label{example}
\end{figure}

\begin{figure}
\centering
\begin{picture}(390,80)(20,-40)

\put(10,0){\line(1,0){94}}   \put(136,0){\line(1,0){136}}   
\put(317,0){\line(1,0){60}}   
 \put(10,0){\circle*{3}}  \put(25,0){\circle*{3}}
 \put(40,0){\circle*{3}}  \put(56,0){\circle*{3}}
 \put(72,0){\circle*{3}}  \put(88,0){\circle*{3}}
 \put(104,0){\circle*{3}}  \put(136,0){\circle*{3}}
 \put(152,0){\circle*{3}}  \put(168,0){\circle*{3}}
 \put(182,0){\circle*{3}}  \put(197,0){\circle*{3}}
 \put(212,0){\circle*{3}}  \put(227,0){\circle*{3}}  \put(242,0){\circle*{3}}
 \put(257,0){\circle*{3}} \put(272,0){\circle*{3}}
 \put(317,0){\circle*{3}}   \put(332,0){\circle*{3}}  \put(347,0){\circle*{3}}
 \put(362,0){\circle*{3}}  \put(377,0){\circle*{3}}

\put(48,0){\oval(16,30)[b]} \put(64,0){\oval(16,30)[b]}
\put(80,0){\oval(16,30)[b]} \put(96,0){\oval(16,30)[b]}
\put(144,0){\oval(16,30)[b]}  \put(160,0){\oval(16,30)[b]}
\put(48,-15){\circle*{3}} \put(64,-15){\circle*{3}}
\put(80,-15){\circle*{3}} \put(96,-15){\circle*{3}}
\put(144,-15){\circle*{3}} \put(160,-15){\circle*{3}}

 \put(227,0){\oval(30,30)[b]}
\put(257,0){\oval(30,30)[b]}  \put(332,0){\oval(30,30)[b]}

 \put(115,0){$\dots$}  

\put(280,0){$\dots$}

\put(41,8){$E_1$}  \put(89,8){$E_4$}  \put(128,8){$E_{n-1}$}  
\put(156,8){$E_n$}  \put(214,8){$E_{n+1}$}  \put(244,8){$E_{n+2}$}
 \put(319,8){$E_{2n}$}

\end{picture}
\caption{The thin diagram $\Phi$ defining the cube $Q$.}\label{phi}
\end{figure}
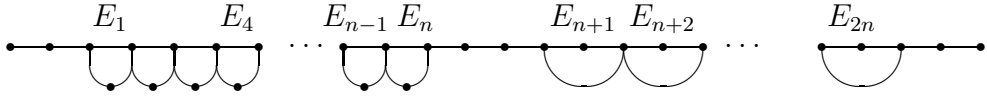

\begin{proposition} There exist points $y\in K(P,w)$ for which the distance $d(y,gy)$ is arbitrarily close to $\sqrt{2}$.

\end{proposition}

\begin{proof} We choose a diagram $U$ for which $W_0=U^{-1}gU$ has the form as on Figure \ref{example}. We now choose a $2n$-dimensional cube $Q=C(U,\Phi)$ based on this diagram, so $k=2n$. Recall the construction and the notation in Subsection 4.1.  There are $3n+7$ edges in $\underline U$ and in $\overline{W}_0$ and we call them $x_1,\dots,x_{3n+7}$, from left to right. Cells $E_1,\dots,E_k$ of $\Phi$ are attached along $\underline U$ as follows. Cells $E_1,\dots,E_n$ are of type $C(x,xx)$ and $E_i$ is attached along $x_{i+2}$. Cells $E_{n+1},\dots,E_{2n}$ are of type $C(xx,x)$ and $E_j$ is attached along $x_{2j+4-n}$ and $x_{2j+5-n}$.
So $E_{2n}$ is attached along $x_{3n+4}$ and $x_{3n+5}$. Diagram $\Phi$ is shown on Figure \ref{phi}. When we attach $\Phi^{-1}$ to $W_0$ from above (and denote the cells of $\Phi^{-1}$ by $E_i'$) only cells $E_n'$ and $E_{2n}'$ cancel. Therefore set $T$ consists of two numbers $n$ and $2n$. The remaining cells of $\Phi^{-1}$ are attached to $W_0$ from above and are called $E_i'$ for $i=1,\dots,n-1$ and for $i=n+1,\dots,2n-1$ as on Figure \ref{example0}. We get $W_1$. We now attach $\Phi$ to $W_1$ along $\underline W_1$ and check which cells of $\Phi$ cancel. Here we denote the cells of $\Phi$ by $F_i$.

The cell $F_1$ attached along the third edge of $\underline W_1$ does not cancel so we have to attach it to get $W$. Number 1 belongs to $L'$. Also $F_{n+1}$ attached along the edges number $(2n+6-n)$ and $(2n+7-n)$ of $\underline W_1$ does not cancel so we have to attach it to get $W$. Number $n+1$ belongs to $L'$. Other cells $F_i$ cancel with suitable cells of $W_1$.
For $i=2,\dots,n$ and for $i=n+2,\dots,2n$  cell $F_i$ cancels with $E_{i-1}'$ so we have $F_i'=E_{i-1}'$. Cell $F_1$ is called $C_{2n+1}$ and cell $F_{n+1}$ is called $C_{2n+2}$. Diagram $W$ has $2n+4$ cells shown on Figure \ref{example1}.

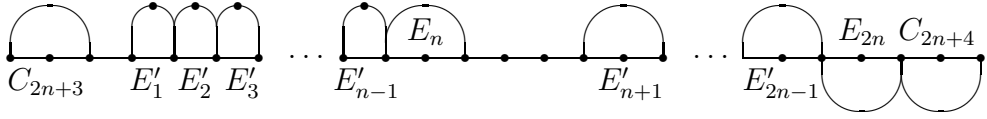
\begin{figure}
\centering
\begin{picture}(390,80)(20,-40)

\put(10,0){\line(1,0){94}}   \put(136,0){\line(1,0){121}}   
\put(287,0){\line(1,0){90}}   
 \put(10,0){\circle*{3}}  \put(25,0){\circle*{3}}
 \put(40,0){\circle*{3}}  \put(56,0){\circle*{3}}
 \put(72,0){\circle*{3}}  \put(88,0){\circle*{3}}
 \put(104,0){\circle*{3}}  \put(136,0){\circle*{3}}
 \put(152,0){\circle*{3}}  \put(167,0){\circle*{3}}
 \put(182,0){\circle*{3}}  \put(197,0){\circle*{3}}
 \put(212,0){\circle*{3}}  \put(227,0){\circle*{3}}  \put(242,0){\circle*{3}}
 \put(257,0){\circle*{3}} \put(302,0){\circle*{3}}
 \put(317,0){\circle*{3}}   \put(332,0){\circle*{3}}  \put(347,0){\circle*{3}}
 \put(362,0){\circle*{3}}  \put(377,0){\circle*{3}}

\put(25,0){\oval(30,40)[t]} \put(167,0){\oval(30,40)[t]}

 \put(332,0){\oval(30,40)[b]}
\put(362,0){\oval(30,40)[b]}

 \put(115,0){$\dots$}  

\put(268,0){$\dots$}
  \put(64,0){\oval(16,40)[t]}
 \put(80,0){\oval(16,40)[t]}
 \put(96,0){\oval(16,40)[t]}
 \put(144,0){\oval(16,40)[t]} 
%  \put(212,0){\oval(30,40)[b]}
 \put(242,0){\oval(30,40)[t]}   \put(302,0){\oval(30,40)[t]}

  \put(64,20){\circle*{3}} 
 \put(80,20){\circle*{3}} 
 \put(96,20){\circle*{3}} 
 \put(144,20){\circle*{3}} 
\put(9,-12){$C_{2n+3}$} 
  \put(55,-12){$E_1'$}
 \put(73,-12){$E_2'$}    \put(91,-12){$E_3'$}
\put(160,7){$E_n$}  
\put(232,-12){$E_{n+1}'$}  \put(132,-12){$E_{n-1}'$} \put(287,-12){$E_{2n-1}'$}
\put(346,7){$C_{2n+4}$}   \put(323,7){$E_{2n}$}
\end{picture}
\caption{Diagram $W_1$ for the counterexample.}\label{example0}
\end{figure}

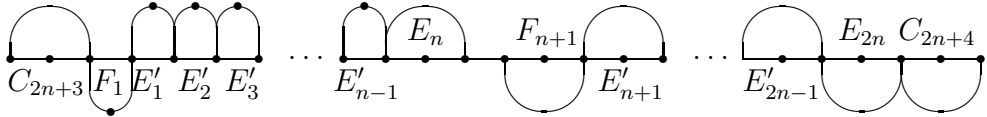
\begin{figure}
\centering
\begin{picture}(390,80)(20,-40)

\put(10,0){\line(1,0){94}}   \put(136,0){\line(1,0){121}}   
\put(287,0){\line(1,0){90}}   
 \put(10,0){\circle*{3}}  \put(25,0){\circle*{3}}
 \put(40,0){\circle*{3}}  \put(56,0){\circle*{3}}
 \put(72,0){\circle*{3}}  \put(88,0){\circle*{3}}
 \put(104,0){\circle*{3}}  \put(136,0){\circle*{3}}
 \put(152,0){\circle*{3}}  \put(167,0){\circle*{3}}
 \put(182,0){\circle*{3}}  \put(197,0){\circle*{3}}
 \put(212,0){\circle*{3}}  \put(227,0){\circle*{3}}  \put(242,0){\circle*{3}}
 \put(257,0){\circle*{3}} \put(302,0){\circle*{3}}
 \put(317,0){\circle*{3}}   \put(332,0){\circle*{3}}  \put(347,0){\circle*{3}}
 \put(362,0){\circle*{3}}  \put(377,0){\circle*{3}}

\put(25,0){\oval(30,40)[t]} \put(167,0){\oval(30,40)[t]}

 \put(332,0){\oval(30,40)[b]}
\put(362,0){\oval(30,40)[b]}

 \put(115,0){$\dots$}  

\put(268,0){$\dots$}
 \put(48,0){\oval(16,40)[b]}   \put(64,0){\oval(16,40)[t]}
 \put(80,0){\oval(16,40)[t]}
 \put(96,0){\oval(16,40)[t]}
 \put(144,0){\oval(16,40)[t]}   \put(212,0){\oval(30,40)[b]}
 \put(242,0){\oval(30,40)[t]}   \put(302,0){\oval(30,40)[t]}

 \put(48,-20){\circle*{3}}   \put(64,20){\circle*{3}} 
 \put(80,20){\circle*{3}} 
 \put(96,20){\circle*{3}} 
 \put(144,20){\circle*{3}} 
\put(9,-12){$C_{2n+3}$}  \put(41,-12){$F_1$} 
  \put(55,-12){$E_1'$}
 \put(73,-12){$E_2'$}    \put(91,-12){$E_3'$}
\put(160,7){$E_n$}   \put(201,7){$F_{n+1}$}
\put(232,-12){$E_{n+1}'$}  \put(132,-12){$E_{n-1}'$} \put(287,-12){$E_{2n-1}'$}
\put(346,7){$C_{2n+4}$}   \put(323,7){$E_{2n}$}
\end{picture}
\caption{Diagram $W$ for the counterexample.}\label{example1}
\end{figure}

Observe that the diagram $W$ on Figure \ref{example1} is a thin diagram, all its cells are independent in the partial ordering of the cells in $W$. Therefore the image of $K_{U_0,V_0}$ in $I^N$ is equal to the whole cube $I^N$ and the path metric in $K_{U_0,V_0}$ is equal to the Euclidean distance in $I^N$. We have $2n$ coordinates in the cube $Q$ and $2n+4$ coordinates in $I^N$.
Only numbers $n$ and $2n$ belong to $T$ and only $1$ and $n+1$  belong to $L'$. 

We choose the point $y\in Q$ with the coordinates 
 $$y_i=\frac{n+1-i}{n+1} \ \ for \ \ i=1,\dots,n-1, \ \ y_n=\frac{n}{n+1}, $$
$$ y_i=\frac{2n+1-i}{n+1} \ \ for \ \ i=n+1,\dots,2n-1, \ \ y_{2n}=\frac{n}{n+1}.$$

We now compute $z=gy$ from Lemma \ref{actioncube} and we compute the difference $|z_i-y_i|$.
Observe that $L\cap T=\{n,2n\}$, $L\cap T'=\mathcal K-\{1,n,n+1,2n\}$, $L'\cap T=\varnothing$ and $L'\cap T'=\{1,n+1\}$.

If $i\in[1,n-2]$ or  $i\in[n+1,2n-2]$ then $C_i=F_{i+1}'$, $i+1\in L\cap T'$ (case (2)), $z_i=y_{i+1}$ and $|z_i-y_i|=\frac{1}{n+1}$.

If $i=n-1$ or $i=2n-1$ then $C_i=F_{i+1}'$, $i+1\in L\cap T$ and we are in case (1). We have $y_i=\frac{2}{n+1}$, $z_i=1-y_{i+1}=\frac{1}{n+1}$ and $|z_i-y_i|=\frac{1}{n+1}$.

If $i=n$ or $i=2n$ then we are in case (3), $z_i=1$ and $|z_i-y_i|=\frac{1}{n+1}$.

If $i=2n+1$ then $C_i=F_1$, if $i=2n+2$ then $C_i=F_{n+1}$ and we are in case (5). We have  $z_{2n+1}=1-y_1=\frac{1}{n+1}$, $z_{2n+2}=1-y_{n+1}=\frac{1}{n+1}$, $y_i=0$ 
 and   $|z_i-y_i|=\frac{1}{n+1}$.

If $i=2n+3$ or $i=2n+4$ then we are in case (3), $y_i=0$  and  $|z_i-y_i|=1$.

 Thus the distance 

$$d(y,gy)=\sqrt{2+\Sigma_{i=1}^{2n+2} \frac{1}{(n+1)^2}}=\sqrt{2+\frac{2}{n+1}}$$

is arbitrarily close to $\sqrt{2}$ for $n$ sufficiently large.

\end{proof}
This concludes the proof of Theorem 1. The isometry $g$ does not act in a semi-simple way on the complex $K(P,w)$.

\par\bigskip

Yael Algom-Kfir,
Yale University,
 Mathematics Department,

New Haven, CT 06520-8283

 yael.algomkfir@yale.edu\par\medskip

Bronislaw Wajnryb,
Department of Mathematics,

Rzeszow University of Technology,
35-959 Rzeszow, Poland,

dwajnryb@prz.edu.pl\par\medskip

Pawel Witowicz,
Department of Mathematics,

Rzeszow University of Technology,
35-959 Rzeszow, Poland,

witowicz@prz.edu.pl
\end{document}